\documentclass{article}

\usepackage[british]{babel}
\usepackage{csquotes}
\usepackage{microtype}
\usepackage[a4paper]{geometry}
\usepackage[useprefix,backref,sorting=none,sortcites=true]{biblatex}
\usepackage{mathtools}
\usepackage{amsfonts,amssymb}
\usepackage{amsthm}
\usepackage{thmtools}
\usepackage{accents}
\usepackage{setspace}
\usepackage{hyperref}
\declaretheorem[name=Theorem]{thm}
\declaretheorem[name=Proposition]{prop}
\declaretheorem[name=Lemma]{lem}
\declaretheorem[name=Conjecture,numbered=no]{conj}
\declaretheorem[name=Corollary,numbered=no]{cor}

\declaretheorem[name=Definition,style=definition]{defn}
\declaretheorem[name=Example,style=definition]{exmp}
\declaretheorem[name=Remark,style=remark,numbered=no]{rem}
\declaretheorem[name=Note,style=remark,numbered=no]{note}

\newcommand{\U}{\mathrm{U}}
\newcommand{\ilc}{\Tilde{\epsilon}}
\newcommand{\lc}{\underaccent{\tilde}{\epsilon}}
\newcommand{\qh}{\Tilde{\star}}
\newcommand{\qhv}{{\underaccent{\tilde}{\star}}}
\newcommand{\Spin}{\mathrm{Spin}}
\newcommand{\vol}{\mathrm{vol}}
\renewcommand{\Re}{\mathop\mathrm{Re}}

\newcommand{\Gtwo}{\mathrm{G}_2}
\newcommand{\SL}{\mathrm{SL}}
\newcommand{\GL}{\mathrm{GL}}
\newcommand{\SO}{\mathrm{SO}}
\newcommand{\SU}{\mathrm{SU}}
\newcommand{\Sp}{\mathrm{Sp}}
\newcommand{\R}{\mathbb{R}}
\newcommand{\hook}{\mathbin\lrcorner}
\DeclareMathOperator{\Hom}{Hom}

\title{A note on 4-forms in 8-dimensions}
\author{Sam Close\thanks{sam.close@nottingham.ac.uk}\\School of Mathematical Sciences, University of Nottingham}

\begin{document}

\maketitle

\begin{abstract}
We present a note on two related questions on 4-forms in 8-dimensions. In particular, we give explicit families of 4-forms, beyond the Cayley and quaternion-K\"ahler forms, which determine a metric via the formula of Karigiannis. Additionally, we discuss an open question of Salamon and Walpuski on non-degenerate 4-forms, and answer it for a more restrictive notion of non-degeneracy.
\end{abstract}

\section{Introduction}
4-forms in 8-dimensions are interesting algebraically and geometrically. From the former point of view, they are the first\footnote{That is, considering the space of \(p\)-forms in \(n\)-dimensions, increasing \(n\), and checking all \(0\leq p\leq n\).} set of forms for which there are infinitely many \(\GL\) orbits \autocite{Ryvkin}. This makes their characterisation highly non-trivial (unlike, for example, 2-forms in any dimension). Work on classifying orbits was first done by Antonyan in the complex case (translated into English in \autocite{Oeding2022}), and recently made real by \textcite{givenEmanuele2025} (a method was proposed to classify these orbits by \textcite{Le2011} previously). 

Most of the interest in this area comes from the Cayley form\footnote{In the literature one may also see this called the \(\Spin(7)\) form, Bonan form, or fundamental 4-form, once \(\Spin(7)\) is established as the group of interest.}: the 4-form stabilised by \(\Spin(7)\), discovered by \textcite{Bonan1966+}. The Cayley form has been extremely well studied at this point, and it is very well understood. It determines an orientation and a metric, albeit via a complicated formula \autocite{Karigiannis2005}. There are also further classifications of the Cayley form relying on comass \autocite{Katz2010,Bangert2009}. For the purposes of this note, we define a Cayley form in the following way.
\begin{defn}
    Let \(V\) be an 8-dimensional real vector space. A 4-form \(\Phi\in\Lambda^4 V^*\) is called a Cayley form if its \(\GL(V)\) stabiliser is \(\Spin(7)\).
\end{defn}

In this short note, we study which 4-forms in 8-dimensions define a metric by the formula of Karigiannis. This is known for the Cayley form and, by \textcite{Fowdar2024}, for the quaternion-K\"ahler form. We highlight that it holds for further families: \emph{the Cayley form is far from unique in defining a non-degenerate Riemannian metric} in this way. This is in contrast to \(\Gtwo\) where there are only two orbits that define a non-degenerate metric (only one Riemannian). We also discuss an open question of Salamon and Walpuski relating to their notion of non-degenerate 4-forms.

\paragraph{Notation} Throughout \(V\) will denote a real 8-dimensional vector space. We will use \(\langle u,v\rangle\) to be the Euclidean inner product of \(u,v\) and \(\|u\|^2 := \langle u,u\rangle\).

\subparagraph{Quasi-Hodge Isomorphism}
We must introduce the quasi-Hodge isomorphism\footnote{This name comes from \autocite{Vaz2016}, but \autocite{Hehl2003} call it a pre-metric duality operator.} briefly. Since we have no metric \textit{a priori}, the only way we can `raise' indices is with the Levi-Civita tensor densities \(\ilc\) and \(\lc\). These allow us to define an isomorphism between \(p\)-vectors (resp.\ \(p\)-forms) and \((n-p)\)-forms (resp.\ \((n-p)\)-vectors) of weight \(-1\) (resp.\ \(+1\)). The definition is much the same as the Hodge dual.
\begin{defn}\label{defn:qh}
    Let \(\alpha\in \Lambda^p W^*\) and \(A\in \Lambda^p W\) where \(W\) is an \(n\)-dimensional real vector space. The quasi-Hodge isomorphism is given as
    \begin{equation*}
        (\qh\alpha)^{a_1a_2\cdots a_{n-p}} = \frac{1}{p!}\ilc^{b_1b_2\cdots b_p a_1a_2\cdots a_{n-p}}\alpha_{b_1b_2\cdots b_p}
    \end{equation*}
    and 
    \begin{equation*}
        (\qhv A)_{a_1a_2\cdots a_{n-p}} = \frac{1}{p!}\lc_{b_1b_2\cdots b_p a_1a_2\cdots a_{n-p}}A^{b_1b_2\cdots b_p}.
    \end{equation*}
\end{defn}
See \autocite{Vaz2016} for an invariant description (i.e., without indices), and \autocite{Schouten1989,Schouten1954} for a more detailed index-theoretic description of the Levi-Civita tensor densities.

\section{On Metric 4-forms}
The Cayley form is well-known to determine a metric. There are at least two intuitive ways of understanding this. The first is simply by noting that the Cayley form is stabilised by \(\Spin(7)\), a subgroup of \(\SO(8)\). The other point of view is spinorial: the space of Cayley forms is 43-dimensional, which can be seen as a metric (36) and a unit spinor (7). Both of these viewpoints are expounded on in \autocite{Krasnov2024}.

The explicit form of the metric determined by a Cayley form is found via a formula of \textcite{Karigiannis2005}, which we recall here formally.
\begin{defn}
    Let \(\{e_i\}_{i=1,2,\ldots,8}\) be a basis of \(V\). Let 
    \begin{equation*}
        A:V\times \Lambda^4 V^*\to \R; \quad (v,\alpha)\mapsto A_\alpha(v) = (v\hook\alpha\wedge\alpha)(e_1,e_2,\ldots,e_7)
    \end{equation*}
    and 
    \begin{equation*}
    \begin{split}
        B:V\times \Lambda^4 V^* &\to S^2\R^7\\
        (v,\alpha)&\mapsto (B_\alpha(v))_{\hat\imath\hat\jmath} = (e_{\hat\imath}\hook v\hook \alpha\wedge e_{\hat\jmath}\hook v\hook\alpha\wedge v\hook \alpha)(e_1,e_2,\ldots, e_7)
    \end{split}
    \end{equation*}
    where \(\hat{\imath},\hat{\jmath}=1,2,\ldots,7\) and w.l.o.g.\ \(v^8\neq0\). Then define
    \begin{equation}
        q:V\times \Lambda^4 V^*\to \R;\quad (v,\alpha)\mapsto \left|\frac{(\det(B_\alpha(v)))^{1/3}}{(A_\alpha(v))^3}\right|.\label{eq:q}
    \end{equation}
\end{defn}
\begin{rem}
Our definition, compared to \autocite{Karigiannis2005}, includes an absolute value to account for the possible discrepancy in orientation between \(u\hook v\hook \alpha\wedge u\hook v\hook \alpha\wedge\alpha \) and \(\alpha\wedge\alpha\).
\end{rem}

The following two results are proven by \textcite{Karigiannis2005}.
\begin{lem}
    \(q\) is absolute homogeneous of order 1 in \(\Lambda^4 V^*\) and homogeneous of order 4 in \(V\).
\end{lem}

\begin{proof}
In the forthcoming, let \(\lambda\in\R\). Sending \(\alpha\mapsto \lambda\alpha\), \(A_\alpha\mapsto \lambda^2 A_\alpha\). Meanwhile, \(B_\alpha\mapsto \lambda^3 B_\alpha\), thus \(\det(B_\alpha) \mapsto \lambda^{21}\det(B_\alpha)\). Thus \(q_\alpha \mapsto |\lambda|q_\alpha\). This proceeds, \textit{mutatis mutandis}, for \(V\), resulting in \(v\mapsto\lambda v\) implying \(q(v)\mapsto \lambda^4 q(v)\).
\end{proof}

\begin{thm}
Let \(\Phi\) be a Cayley form on \(V\). Then \(q_\Phi(v)\propto \|v\|^4\).
\end{thm}
\begin{proof}
     This follows from the algebra a Cayley form satisfies: see \autocite{Karigiannis2005}.
\end{proof}
\textit{Prima facie}, it is not obvious why \(q_\Phi(v)\) is proportional to the norm. Equally, there is no reason \textit{a priori} why this should only work for the Cayley form: there appears to be nothing required of the 4-form in this expression, other than it must determine a top-form.

Since the norm is positive definite in the case of a Cayley form, we can meaningfully extract the metric by taking the root and polarising
\begin{equation*}
    g_\Phi(u,v) = \frac{1}{4}\left(\sqrt{q_\Phi(u+v)}-\sqrt{q_\Phi(u-v)}\right).
\end{equation*}
One can use the same formula for \(q\) on the split Cayley form \(\Phi^\prime\) (stabilised by \(\Spin(4,3)\)) to obtain the square of an isotropic quadratic form. A metric cannot be extracted from \(q_{\Phi^\prime}\) in the same way as \(q_\Phi\) due to a choice of branch when taking the root; however it is still accepted that this 4-form defines a split signature metric. This leads on to the following definition.

\begin{defn}\label{def:Metric4Form}
    A 4-form \(\alpha\) is \emph{metric} if \(q_\alpha\) is proportional to the square of a non-degenerate quadratic form.
\end{defn}
In this sense, both Cayley forms (\(\Phi\) stabilised by \(\Spin(7)\) and \(\Phi^\prime\) stabilised by \(\Spin(4,3)\)) are metric 4-forms. 

\begin{prop}\label{thm:CayleyNotUnique}
The Cayley form is not unique in defining a non-degenerate Riemannian metric via \eqref{eq:q}.
\end{prop}
\begin{note}
    After publication, we became aware of a paper by \textcite{Fowdar2024} which showed that Karigiannis's formula holds for \(\Sp(2)\Sp(1)\) as well. Indeed, this formula will give something proportional to the norm essentially whenever the stabiliser acts transitively on \(S^7\): the challenge is showing that the constant of proportionality is non-zero.
\end{note}
In the proof, we discuss arguably the simplest other example. Note that the calculations involving \(q\) were obtained by Mathematica \autocite{Mathematica}.
\begin{proof}
    Suppose we have an \(\SU(4)\) structure, i.e., a K\"ahler 2-form
    \begin{equation*}
        \omega = e^{12} + e^{34}+ e^{56}+e^{78}
    \end{equation*}
    and the holomorphic volume form
    \begin{equation*}
        \Omega = (e^1 + ie^2)\wedge(e^3+ie^4)\wedge(e^5+ie^6)\wedge (e^7+ ie^8).
    \end{equation*}
    It is well-known that the Cayley form can be expressed in terms of these forms as
    \begin{equation*}
        \Phi = \frac{1}{2}\omega\wedge\omega + \Re(\Omega).
    \end{equation*}
    Instead, we can consider a 1-parameter family of 4-forms
    \begin{equation}
        \beta_t = \frac{1}{2}\omega\wedge\omega + t\Re(\Omega).\label{eq:alphat}
    \end{equation}
    \(\beta_t\wedge\beta_t=2(3+4t^2)e^1\wedge e^2\wedge\cdots \wedge e^8\) from which follows
    \begin{equation*}
        A_{\beta_t}(v) = -(3+4t^2)v^8.
    \end{equation*}
    Calculating \(B_{\beta_t}\) is significantly more difficult, but one can show that
    \begin{equation}
        q_{\beta_t}(v) = 6^{7/3}\frac{t^4}{(3+4t^2)^3} \|v\|^4.\label{eq:gbeta}
    \end{equation}
    Thus there exists a 1-parameter family of 4-forms stabilised by \(\SU(4)\) which define a non-degenerate Riemannian metric, and which intersect the \(\Spin(7)\) orbit when \(t=\pm1\) (for \(t\neq0,\pm1\), this is orbit 11 in \autocite{givenEmanuele2025}, which they term Calabi--Yau). When \(t=0\), the stabiliser becomes \(\Sp(4,\R)\), and the form is no longer metric.
\end{proof}

The example given is sufficient for the proof, but it is interesting to provide a larger family of examples, which contains \(\beta_t\).
\begin{exmp}[Self-Dual Subspaces]\label{exmp:sd}
   Let \(V=W_1\oplus W_2\), with \(E^{1,2,3,4}\) a basis for \(W_1^*\) and \(F^{1,2,3,4}\) a basis for \(W_2^*\). Notating \(E^{ij}=E^i\wedge E^j\) and \(F^{ij}=F^i\wedge F^j\), define the 2-forms
   \begin{equation*}
   \begin{split}
    \sigma^1 &= E^{12}+ E^{34} + F^{12} + F^{34},\\
    \sigma^2 &= E^{13}+ E^{42}+F^{13}+ F^{42},\\
    \sigma^3 &=  E^{14}+ E^{23}+ F^{14}+ F^{23}.
   \end{split}
   \end{equation*}
   Then define the 3-parameter family of 4-forms
   \begin{equation*}
   \eta = \frac{1}{2}(a_1\sigma^1\wedge\sigma^1 + a_2\sigma^2\wedge\sigma^2+a_3\sigma^3\wedge\sigma^3).
   \end{equation*}
   \(\eta\) captures a number of interesting forms for particular values of the parameters. When \(a_1=a_2=a_3=1\), \(\eta\) is the Kraines form \autocite{Kraines1965,Kraines1966} (or the quaternion-K\"ahler form in the language of \autocite{givenEmanuele2025}; their 7th orbit), stabilised by \(\Sp(2)\Sp(1)\). When \((a_1,a_2,a_3)\) is a permutation or overall sign reversal of \((1,1,-1)\), \(\eta\) is the Cayley form (orbit 5 in \autocite{givenEmanuele2025}). The symmetry of the Kraines form can be broken. When \(a_1=a_2\neq a_3\) but not the Cayley case (and with all \(a_i\) non-zero and \(a_1+a_2+a_3\neq0\)), the \(\Sp(1)\) of the stabiliser breaks into \(\U(1)\) (this is orbit 13 in \autocite{givenEmanuele2025}). When the \(a_i\) are pairwise distinct and non-zero, with \(a_1+a_2+a_3\neq0\) and \(a_i+a_j\neq0\) for all \(i\neq j\), the symmetry breaks to just \(\Sp(2)\), which corresponds to the hyper-K\"ahler 4-form (orbit 19 in \autocite{givenEmanuele2025}). If instead the \(a_i\) are pairwise distinct with \(a_i+a_j=0\) for one pair \(i\neq j\) (and the remaining parameter non-zero), the stabiliser enlarges to \(\SU(4)\): this is the family \(\beta_t\) of \autoref{thm:CayleyNotUnique}, with \((a_1,a_2,a_3)=(1,t,-t)\), since \(\Re(\Omega)=\frac{1}{2}(\sigma^2\wedge\sigma^2-\sigma^3\wedge\sigma^3)\) upon identifying \(E^i=e^i\) and \(F^i=e^{i+4}\). The stabiliser also enlarges, to a non-compact group, when some \(a_i=0\) or when \(a_1+a_2+a_3=0\); in these cases \(\eta\) is not metric. The relationship between these forms is elaborated on by \textcite{Bryant1989a}. In general, we have
   \begin{equation*}
   q_\eta(v) = \frac{6^{7/3}a_1^2a_2^2a_3^2|a_1+a_2+a_3|}{(3(a_1^2+a_2^2+a_3^2)+2(a_1a_2+a_2a_3+a_3a_1))^3}\|v\|^4.
   \end{equation*}
   For the Kraines form, this recovers the result of \textcite[Proposition~2.7]{Fowdar2024}.
\end{exmp}
In the proof of \autoref{thm:CayleyNotUnique}, we construct the 4-form \(\beta_t\) using the data of an \(\SU(4)\) structure: the K\"ahler form and the holomorphic volume form. These data already define a Riemannian metric, once one constructs the almost complex structure. Similarly, in \autoref{exmp:sd}, we construct the 4-form \(\eta\) using two triples of 2-forms. These define a (product) metric via an Urbantke-like formula:
\begin{equation*}
g(u,v) \vol = \epsilon_{ijk} u\hook \sigma^i\wedge v\hook \sigma^j \wedge\sigma^k \wedge\eta    
\end{equation*}
where \(\epsilon_{ijk}\) is the usual totally antisymmetric symbol in 3-dimensions. Perhaps then it is not surprising that the 4-forms define metrics, and nor necessarily should it be: they are stabilised by subgroups of \(\SO(8)\), which is one of the same justifications for why the Cayley form defines a metric. The point we want to highlight with this note is that the formula of Karigiannis applies well beyond the Cayley form, and that one can start from a 4-form abstract of the data used to construct it (i.e., just start with an orbit representative) to recover the same metric.

\section{On the Non-Degeneracy of 4-forms}\label{sec:nondegen}
In dimension seven, we have a notion of a non-degenerate 3-form: that is, a 3-form \(\varphi\) is non-degenerate iff it determines a non-degenerate metric via the formula
\begin{equation}
    g_\varphi(u,v)\vol_\varphi = u\hook\varphi\wedge v\hook\varphi\wedge\varphi.\label{eq:G2metric}
\end{equation}
It is well-known that there are only two non-degenerate orbits of 3-forms in 7-dimensions, corresponding to \(\Gtwo\) and \(\Tilde{\mathrm{G}}_2\), the split \(\Gtwo\) (see \autocite{Agricola2008} for an interesting historical discussion).

With 4-forms in 8-dimensions, the notion of non-degeneracy is not as clear. In this section, we focus on an open question posed by \textcite{Salamon2017}. To state it, we must first recall their definition of non-degeneracy.

\begin{defn}\label{def:SWnd}
   A 4-form \(\alpha\in\Lambda^4 V^*\) is called \emph{non-degenerate} iff for all linearly independent triples \(u,v,w\in V\), \(\exists\ x\in V\) such that \(\alpha(u,v,w,x)\neq0\). Otherwise, it is called \emph{degenerate}.
\end{defn}

There is also a notion of non-degeneracy in multisymplectic geometry \autocite{Ryvkin}: a \(p\)-form \(\alpha\) is called non-degenerate if \(\hook\alpha:V\to \Lambda^{p-1}V^*\) is injective. In 8-dimensions with 4-forms, the two notions are related, but not equivalent. One can show the following. 
\begin{lem}
    Non-degeneracy in the sense of \autoref{def:SWnd} implies multisymplectic non-degeneracy, i.e.,
    \begin{equation*}
        \text{non-degenerate in the sense of \autoref{def:SWnd}}\implies \text{non-degenerate in the multisymplectic sense}.
    \end{equation*}
\end{lem}
\begin{proof}
It is actually easier to prove the contrapositive. Suppose \(\alpha\) is degenerate in the multisymplectic sense. Then choose a \(u,v\) with \(u\neq v\) such that \(u\hook\alpha=v\hook\alpha\), or \((u-v)\hook\alpha=0\). Choose two other vectors \(w,x\) such that \(u-v,w,x\) are linearly independent. Then there does not exist a \(y\) such that \(\alpha(u-v,w,x,y)\neq0\). Thus \(\alpha\) is degenerate in the sense of \autoref{def:SWnd}. 

To prove that this is not an if and only if statement, let \(V=W_1\oplus W_2\) split as in \autoref{exmp:sd}
\begin{equation}
    \xi = \vol_{W_1}+\vol_{W_2} = E^1\wedge E^2\wedge E^3\wedge E^4 + F^1\wedge F^2\wedge F^3 \wedge F^4\label{eq:SLSLform}
\end{equation}
\(\xi\) is stabilised by \(\SL(4,\R)\times\SL(4,\R)\). Then \(v,w\in V\) split as
\begin{equation*}
    v=v_1+v_2 \ \text{and}\ w=w_1+w_2.
\end{equation*}
Taking the interior product,
\begin{equation*}
    (v-w)\hook \xi = (v_1-w_1)\hook \vol_{W_1} + (v_2-w_2)\hook \vol_{W_2}
\end{equation*}
Volume forms are injective, and these 3-forms are only defined on each space \(W_1,W_2\), thus this can only vanish if \(v=w\), so \(\xi\) is non-degenerate in the multisymplectic sense. To show that it is degenerate in the sense of \autoref{def:SWnd}, it suffices to check
\begin{equation*}
    u = E_1,\ v=E_2 ,\ w=F_3 
\end{equation*}
where the lower indices mean these are dual to \(E^{1,2,3,4},F^{1,2,3,4}\). One can easily see that \(w\hook v\hook u\hook \xi=0\), thus this is degenerate.
\end{proof}
\begin{note}
\citeauthor{Salamon2017} showed that the Cayley form is non-degenerate in their sense. \textcite{Kennon2021} showed independently that the Cayley form is non-degenerate in the sense of multisymplectic geometry.
\end{note}
\begin{note}
 From hereon when we refer to non-degeneracy of forms, we will mean in the sense of \autoref{def:SWnd}. We will also require the notion of non-degeneracy of bilinear forms, though this should be obvious from context.
\end{note}
\citeauthor{Salamon2017} ask whether every non-degenerate 4-form has non-vanishing self-wedge: they pose this as an open question, reporting that they could neither find a counterexample nor a proof. We phrase the affirmative answer as a conjecture.
\begin{conj}[after Salamon-Walpuski]\hypertarget{conj:MainConj}{}
Let \(\alpha\) be a 4-form on \(V\). If \(\alpha\wedge\alpha=0\), then \(\alpha\) is degenerate.
\end{conj}

\begin{rem}
Certainly this is \emph{not} an if and only if statement: one can trivially find degenerate 4-forms with non-vanishing top-form. The easiest example is the form stabilised by \(\SL(4,\R)\times \SL(4,\R)\), as shown above in \eqref{eq:SLSLform}.
\end{rem}
It is not a trivial exercise to find non-degenerate 4-forms; however, the problem is made more tractable by a theorem of Salamon and Walpuski (not formulated as such in \autocite{Salamon2017}).
\begin{thm}\label{thm:nondegenequiv}
    Let \(\alpha\in \Lambda^4V^*\). Then the following are equivalent.
    \begin{itemize}
        \item \(\alpha\) is non-degenerate.
        \item \(u\hook v\hook \alpha\) is a symplectic form on \(V/\mathrm{span}\{u,v\}\) for linearly independent \(u,v\in V\).
        \item \(u\hook v\hook \alpha\wedge u\hook v\hook \alpha\wedge \alpha \) is non-zero for all linearly independent \(u,v\in V\).
    \end{itemize}
\end{thm}
\begin{proof}
    This is stated, without proof, in the paragraph following the proof of Lemma~7.4 in \autocite{Salamon2017}. The first two statements are equivalent by \autoref{def:SWnd}, and the last two are equivalent since contracting the 8-form in the third statement with \(u\) and \(v\) gives the cube of the 2-form \(u\hook v\hook\alpha\).
\end{proof}
Via the third characterisation, there is another way to state non-degeneracy (the forward implication is used in the proof of Theorem~7.8 in \autocite{Salamon2017}).
\begin{lem}
    Let \(\alpha\in\Lambda^4 V^*\). Then \(\alpha\) is non-degenerate iff \(\forall v\in V\) for \(v\neq0\), \((v\hook\alpha)|_{W^*}\) is the 3-form stabilised by \(\Gtwo\), \(\phi\), where \(W=V/\mathrm{span}\{v\}\).
\end{lem}
\begin{proof}
    Recall that with a preferred vector chosen (\(v\), in this case), any form can be split into
    \begin{equation}
       \alpha = \xi\wedge\beta +\gamma\label{eq:alphasplit}
    \end{equation}
    where \(\xi(v)=1\) w.l.o.g., \(v\hook\beta=0\) and \(v\hook\gamma=0\). Fix a \(v\), and consider
    \begin{equation*}
        u\hook v\hook\alpha\wedge u\hook v\hook \alpha\wedge\alpha.
    \end{equation*}
    Expanding in \eqref{eq:alphasplit},
    \begin{equation*}
        u\hook \beta\wedge u\hook\beta \wedge(\xi\wedge \beta+\gamma).
    \end{equation*}
    This is still an 8-form. We take the interior product once more with \(v\) to obtain a 7-form
    \begin{equation*}
        u\hook\beta\wedge u\hook\beta \wedge\beta.
    \end{equation*}
    Restricting to \(W^*\), this is exactly the formula for the metric determined by a 3-form in 7-dimensions (cf.\ \eqref{eq:G2metric}). In particular, this is non-zero for all \(u\) not proportional to \(v\) iff \(\beta\) is a 3-form stabilised by \(\Gtwo\): in the split case, the metric \eqref{eq:G2metric} is indefinite and so admits null vectors.
\end{proof}

The most fruitful approach is via this last characterisation in \autoref{thm:nondegenequiv}. Let us make this more formal by means of an area metric.

\begin{defn}
    An area metric on \(V\) is a map \(\Lambda^2 V\times \Lambda^2 V\to \R\) which is bilinear and symmetric. Equivalently, it is an element of \(S^2\Lambda^2 V^*\subset V^*\otimes V^*\otimes V^*\otimes V^*\).
\end{defn}
\begin{note}
    At this stage, we do not require non-degeneracy of this area metric. It is also argued in the literature that an area metric should satisfy the algebraic Bianchi identity (at least in 4-dimensions; see \autocite{Borissova2023} for a physical motivation): we will also not assume this, but we will come to discuss this in the context of irreducible tensors later on.
\end{note}
Before introducing the central area metric of our study, we must first define a suitable scalar density. This is because, as mentioned in the introduction, the only way we can `raise' indices of our 4-form is by use of the quasi-Hodge isomorphism. This densitises the expression, thus we need a scalar density to dedensitise it to get a \textit{bona fide} area metric.

Using a 4-form in 8-dimensions, there are \emph{at least} two scalar densities which are, \textit{prima facie}, independent. The first is quite obvious: \(\qh(\alpha\wedge\alpha)\) for some 4-form \(\alpha\). The second is less obvious, and requires us to define some additional tools.

\paragraph{Tensors as Multilinear Maps}
Remark that \(\alpha\) can be viewed as a multilinear map \(V\times V\times V\times V\to \R\) or equivalently \(\Lambda^2 V\times \Lambda^2 V\to \R\). Partial evaluation on a bivector \(X\) gives a 2-form, thus we can also think of \(\alpha\) as a map \(\Lambda^2 V\to \Lambda^2 V^*\) or an element of \(\Hom(\Lambda^2 V,\Lambda^2 V^*)\). To make this interpretation clearer, we define the maps
\begin{equation*}
    \rho: S^2\Lambda^2 V\to \Hom(\Lambda^2 V^*,\Lambda^2 V), \ \rho^*:S^2\Lambda^2 V^*\to\Hom(\Lambda^2 V,\Lambda^2 V^*).
\end{equation*}
Homomorphisms of vector spaces can be identified with matrices, and since \(\dim(\Lambda^2 V)=\dim(\Lambda^2 V^*)\), \(\rho^*(\alpha)\) is a square matrix; thus the determinant of \(\rho^*(\alpha)\) is well defined. This is a scalar density of degree 28 in \(\alpha\). We show in \autoref{adx:ScalarDensities} that \(|\det(\rho^*(\alpha))|^{1/14}\) is actually an absolute scalar density of weight 1, of degree 2 in \(\alpha\).

There is a good reason to choose \(\qh(\alpha\wedge\alpha)\): since \(|\det(\rho^*(\alpha))|^{1/14}\) is an absolute scalar density, if we use this to de-densitise, the tensor will become a twisted tensor, changing sign under orientation reversing changes of basis. To simplify matters, and place them back on equal footing, we will assume only orientation preserving changes of basis. Since these are now, in general, independent, we have no reason \textit{a priori} to choose one over the other. For that reason, in the forthcoming, we write \(\Tilde{\chi}_\alpha\) to be one of \(\qh(\alpha\wedge\alpha)\) or \(|\det(\rho^*(\alpha))|^{1/14}\).

\begin{note}
    There do exist non-zero 4-forms for which both of these scalar densities vanish. We will not have to worry about such 4-forms.
\end{note}

Now we can define the main area metric we need.

\begin{defn}\label{def:GAreaMet}
Let \(\alpha\) be a 4-form on \(V\). Define \(G_\alpha\) an area metric, given by
    \begin{equation*}
        G_\alpha(X,Y) = \frac{1}{\Tilde{\chi}_\alpha}\qh(X \hook \alpha\wedge Y\hook \alpha\wedge\alpha)
    \end{equation*}
for \(X,Y\in\Lambda^2 V\).
\end{defn}
To determine the signature of \(G_\alpha\), we note that as above we can think of \(G_\alpha\) as an element of \(\Hom(\Lambda^2 V,\Lambda^2 V^*)\), and thus \(\rho^*(G_\alpha)\) as a linear map, described above. One can then easily show the following.
\begin{lem}\label{lem:MatrixGalpha}
    \(\rho^*(G_\alpha)\) is congruent (in the sense of linear algebra) to \(\rho(\qh\alpha)\) and thus, if \(\det(\rho^*(\alpha))\neq0\), \(G_\alpha\) is indefinite.
\end{lem}
\begin{proof}
    Trivial. It is an easy exercise in index manipulations to show that
    \begin{equation*}
    (G_\alpha)_{ijkl} = \frac{1}{4\Tilde{\chi}_\alpha} \alpha_{ijab}(\qh\alpha)^{abcd}\alpha_{cdkl}
    \end{equation*}
    where the indices of \(\alpha\) are raised by the action of \(\qh\) (as per \autoref{defn:qh}). By use of \(\rho\) and \(\rho^*\), the above can be realised as the linear map
    \begin{equation*}
        \rho^*(G_\alpha) \propto \frac{1}{\Tilde{\chi}_\alpha} \rho^*(\alpha)\rho(\qh\alpha)\rho^*(\alpha)
    \end{equation*}
    with juxtaposition denoting composition. \(\rho^*(\alpha)=\rho^*(\alpha)^T\) so, if \(\det(\rho^*(\alpha))\neq0\), then \(\rho^*(G_\alpha)\) and \(\rho(\qh\alpha)\) are congruent. By Sylvester's Law of Inertia, this implies \(\rho(\qh\alpha)\) and \(\rho^*(G_\alpha)\) have the same signature: the former has indefinite signature, hence the claim follows.
\end{proof}
We also see here that non-degeneracy of \(G_\alpha\) is equivalent to \(\det(\rho^*(\alpha))\neq0\).

It is a well-known fact of representation theory \autocite{Russo} that the space \(S^2\Lambda^2 V^*\) splits under the Bianchi antisymmetrisation map into
\begin{equation}
    S^2\Lambda^2 V^* = \Lambda^4 V^*\oplus S^2_B \Lambda^2 V^*\label{eq:BianchiSplit}
\end{equation}
where \(S^2_B\Lambda^2 V^*\) are tensors of \(S^2\Lambda^2 V^*\) satisfying the algebraic Bianchi identity. This decomposition is independent of dimension and other structure. Yet in our case, we have a canonical 4-form which builds our area metric. Thus we should expect the two to be somehow related via this splitting, and indeed they are.
\begin{thm}\label{thm:GSplit}
    Let \(\alpha\) be a 4-form. \(G_\alpha\) decomposes (in the manner of \eqref{eq:BianchiSplit}) as
    \begin{equation*}
        G_\alpha = \frac{\qh(\alpha\wedge\alpha)}{2\Tilde{\chi}_\alpha}\alpha - H_\alpha
    \end{equation*}
    where \(H_\alpha\in S^2_B\Lambda^2 V^*\), given explicitly by its action on simple bivectors (or quadruples of vectors)
    \begin{equation*}
        H_\alpha(u,v,w,x) = \frac{1}{\Tilde{\chi}_\alpha}\qh(u\hook v\hook\alpha\wedge w\hook \alpha\wedge x\hook\alpha).
    \end{equation*}
\end{thm}
\begin{proof}
    Expanding out
    \begin{equation*}
        x\hook(u\hook v\hook\alpha \wedge w\hook \alpha\wedge\alpha) = x\hook u\hook v\hook\alpha\wedge w\hook\alpha\wedge\alpha + u\hook v\hook \alpha\wedge x\hook w\hook \alpha \wedge\alpha - u\hook v\hook\alpha\wedge w\hook\alpha\wedge x\hook\alpha.
    \end{equation*}
    The l.h.s.\ vanishes by degree. Noting our definitions of \(G_\alpha\) and \(H_\alpha\), we see
    \begin{equation*}
        0 = x\hook u\hook v\hook\alpha\wedge w\hook\alpha\wedge\alpha - \Tilde{\chi}_\alpha \qhv G_\alpha(u,v,w,x) - \Tilde{\chi}_\alpha \qhv H_\alpha(u,v,w,x).
    \end{equation*}
    To deal with the remaining term, expand out
    \begin{equation*}
        w\hook(x\hook u\hook v\hook\alpha \wedge \alpha\wedge \alpha) = \alpha(v,u,x,w) \alpha\wedge\alpha - 2 x\hook u \hook v\hook\alpha\wedge w\hook\alpha\wedge\alpha.
    \end{equation*}
    Once again, the l.h.s.\ vanishes by degree, thus
\begin{equation*}
    x\hook u\hook v\hook\alpha\wedge w\hook\alpha\wedge\alpha = \frac{1}{2}\alpha(u,v,w,x)\alpha\wedge\alpha.
\end{equation*}
Plugging this back in, 
\begin{equation*}
    0 =\frac{1}{2} \alpha(u,v,w,x)\alpha\wedge\alpha - \Tilde{\chi}_\alpha \qhv G_\alpha(u,v,w,x) - \Tilde{\chi}_\alpha \qhv H_\alpha(u,v,w,x).
\end{equation*}
Finally, applying \(\qh\) and dividing by \(\Tilde{\chi}_\alpha\), we arrive at
\begin{equation*}
    G_\alpha = \frac{1}{2}\frac{\qh(\alpha\wedge\alpha)}{\Tilde{\chi}_\alpha}\alpha - H_\alpha.
\end{equation*}
We should also show that \(H_\alpha\) satisfies the Bianchi identity. Consider
\begin{equation*}
    u\hook (v\hook \alpha\wedge x \hook \alpha \wedge w \hook \alpha).
\end{equation*}
This vanishes by degree, but expanding out we have
\begin{equation*}
    0 = u\hook v\hook\alpha \wedge x\hook \alpha\wedge w\hook\alpha  - v\hook\alpha\wedge u\hook x\hook \alpha \wedge w\hook \alpha + v\hook \alpha\wedge x\hook\alpha\wedge u\hook w\hook \alpha
\end{equation*} 
or
\begin{equation*}
    H_\alpha(u,v,x,w) + H_\alpha(u,x,w,v) + H_\alpha(u,w,v,x)=0,
\end{equation*}
which is the Bianchi identity.
\end{proof}
Note that when evaluating \(G_\alpha(u,v,u,v)\) to determine non-degeneracy via \autoref{thm:nondegenequiv}, this is only sensitive to \(H_\alpha(u,v,u,v)\) by the total antisymmetry of \(\alpha\). As such, understanding non-degeneracy can be reframed in terms of understanding \(H_\alpha\). In particular, if \(H_\alpha\) is definite, then \(G_\alpha(u,v,u,v)\) will be non-zero. This leads us to the following definition.

\begin{defn}\label{defn:StrongND}
    Let \(\alpha\) be a 4-form on \(V\). We say \(\alpha\) is \emph{strongly non-degenerate} if \(H_\alpha\) is definite.
\end{defn}

We say \emph{strong} since strong non-degeneracy clearly implies non-degeneracy. Immediately we have the following corollary.

\begin{cor}\hypertarget{cor:noSDQzero}{}
    There are no strongly non-degenerate 4-forms with vanishing self-wedge.
\end{cor}
\begin{proof}
    Suppose \(\alpha\wedge\alpha=0\). Then, from \autoref{thm:GSplit}, \(G_\alpha =-H_\alpha \). But \(G_\alpha\) is indefinite by \autoref{lem:MatrixGalpha}, so \(H_\alpha\) is also indefinite, thus \(\alpha\) cannot be strongly non-degenerate.
\end{proof}
We end by summarising the state of the question. Salamon and Walpuski ask if
\begin{equation*}
    \alpha\wedge\alpha=0 \implies \alpha\ \text{degenerate?}
\end{equation*}
By introducing the notion of strong non-degeneracy (see \autoref{defn:StrongND}), we have shown that 
\begin{equation*}
     \text{strong non-degeneracy}\implies \text{non-degeneracy}
\end{equation*}
and that 
\begin{equation*}
    \alpha\text{ is strongly non-degenerate}\implies \alpha\wedge\alpha\neq0
\end{equation*}
by \hyperlink{cor:noSDQzero}{the above corollary}. Via a counterexample, one can prove the following.
\begin{prop}\label{stronndandnd}
    Strong non-degeneracy and non-degeneracy are not equivalent. 
\end{prop}
\begin{proof}
    Take \(\beta_t\) as defined in \eqref{eq:alphat}. Calculating \(G_{\beta_t}(u,v,u,v)\) (and choosing \(\Tilde{\chi}_{\beta_t}=\qh(\beta_t\wedge\beta_t)\)) one finds
    \begin{equation}
    G_{\beta_t}(u,v,u,v) = \frac{3}{3+4t^2} (  (1-t^2) \omega(u,v)^2 +t^2 ( \|u\|^2 \|v\|^2 - \langle u,v\rangle^2 )) = -H_{\beta_t}(u,v,u,v).\label{eq:Gbetat}
    \end{equation}
    This is strongly non-degenerate for \(t\in(1/2,\sqrt{3/2})\cup(-\sqrt{3/2},-1/2)\): at \(t=\pm 1/2\) and \(t=\pm\sqrt{3/2}\), \(H_{\beta_{t}}\) becomes semidefinite (and, in fact, \(\det(\rho^*(H_{\beta_t}))=0\) at these points). Consider one of these points, \(t=1/2\). Here,
    \begin{equation*}
        G_{\beta_{1/2}}(u,v,u,v) = \frac{3}{16}(3 \omega(u,v)^2 + (\|u\|^2\|v\|^2-\langle u,v\rangle^2) )>0
    \end{equation*}
    thus \(\beta_{1/2}\) is non-degenerate. Hence \(\beta_{1/2}\) is non-degenerate but not strongly non-degenerate.
\end{proof}
It is also worth noting the relationship between metric 4-forms and non-degenerate 4-forms. Since \(\alpha\wedge\alpha\) appears in the denominator of \(q_\alpha\) \eqref{eq:q}, it is a necessary (but not sufficient) condition for a 4-form to have non-vanishing self-wedge in order to define a metric, thus
\begin{equation*}
\alpha\text{ is metric}\implies \alpha\wedge\alpha\neq0.
\end{equation*}
As discussed in the proof of \autoref{stronndandnd}, \(\beta_t\) is strongly non-degenerate for \(t\in(1/2,\sqrt{3/2})\cup(-\sqrt{3/2},-1/2)\); however, outside of this interval, \(\beta_t\) is still metric (per \eqref{eq:gbeta}), so long as \(t\neq0\). Thus being metric does not imply strong non-degeneracy.

In fact, being metric does not even imply non-degeneracy. The Kraines form, i.e., \(\eta\) of \autoref{exmp:sd} with \(a_1=a_2=a_3=1\), is metric, but it is degenerate: writing \(E_i,F_i\) for the basis dual to \(E^i,F^i\) and taking
\begin{equation*}
    u = E_1+E_2+E_3+F_3,\quad v=F_4,\quad w=-E_1+E_3-E_4+F_1,
\end{equation*}
one checks that \(\eta(u,v,w,\cdot)=0\) (cf.\ Lemma~2.8 of \autocite{Fowdar2024}). On the other hand, \eqref{eq:Gbetat} shows that \(\beta_t\) is non-degenerate for all \(t\neq0\), since \(\omega(u,v)^2\leq\|u\|^2\|v\|^2-\langle u,v\rangle^2\) by Wirtinger's inequality \autocite[Lemma~2.3]{Bangert2009}. Thus the metric 4-forms include forms which are strongly non-degenerate, forms which are non-degenerate but not strongly so, and forms which are degenerate.

\subsection*{Acknowledgements}
S.C.\ is grateful to Jacek Rzemieniecki, Joseph Duthie, and Tathagata Ghosh for encouragement to write this note, to Adam Shaw and Henrik Naujoks for their comments, and to Kirill Krasnov for his detailed comments, feedback, and support throughout. S.C.\ is indebted to Alex Gu for pointing out that a conjecture in a previous version of this document is not true, and the counterexample is now used as a proof of \autoref{stronndandnd}. S.C.\ is supported by the Engineering and Physical Sciences Research Council EP/W524402/1.

\subsection*{Declaration of Artificial Intelligence Use}
Claude Fable 5.2 was used to audit this article: it provided some corrections and clarifications on mathematics, and drafted and reworded some paragraphs. It also identified that the Kraines form is metric but degenerate, and that \(\beta_t\) is in the \(\eta\) family and is non-degenerate for all \(t\neq0\). Most importantly, it highlighted the existence of \autocite{Fowdar2024} to us, and led to us adding a note to this article after \autoref{thm:CayleyNotUnique}. The author takes full responsibility for the content of this article.

\printbibliography

\appendix
\section{Canonical Scalar Densities}\label{adx:ScalarDensities}
There are \emph{at least} two independent scalar densities one can construct using a canonical 4-form (using \(\alpha\) throughout). The first is obvious: \(\qh(\alpha\wedge\alpha)\). The other is less obvious.
\begin{lem}
Let \(\alpha\) be a 4-form. \(|\det(\rho^*(\alpha))|^{1/14}\) is an absolute scalar density of weight \(+1\).
\end{lem}
\begin{proof}
    Under a change of basis, \(\alpha\) transforms as
    \begin{equation*}
        \alpha_{abcd}\mapsto g_a{}^i g_b{}^j g_c{}^k g_d{}^l \alpha_{ijkl}.
    \end{equation*}
    Notating \(\bigwedge^2 g\) the induced action of \(g\) on \(\Lambda^2 V^*\), we can express the above as
    \begin{equation*}
        \alpha_{abcd}\mapsto (\bigwedge{}^2 g)_{ab}{}^{ij} \alpha_{ijkl} (\bigwedge{}^2 g)_{cd}{}^{kl} .
    \end{equation*}
    Now using \(\rho^*\) as previously defined,
    \begin{equation*}
        \rho^*(\alpha) \mapsto \rho^*(\bigwedge{}^2 g) \rho^*(\alpha) \rho^*(\bigwedge{}^2 g)^T
    \end{equation*}
    Taking the determinant,
    \begin{equation*}
        \det(\rho^*(\alpha)) \mapsto \det(\rho^*(\bigwedge{}^2 g))^2 \det(\rho^*(\alpha))
    \end{equation*}
    Using the Sylvester-Franke theorem \autocite{Tornheim1952,Flanders1953},
    \begin{equation*}
        \det(\rho^*(\alpha)) \mapsto \det(g)^{14}\det(\rho^*(\alpha)).
    \end{equation*}
    \(\det(\rho^*(\alpha))\) can be positive or negative, we must take the absolute value. Thus \(|\det(\rho^*(\alpha))|^{1/14}\) is an absolute scalar density of weight \(+1\).
\end{proof}
Note that both of these scalar densities are homogeneous of order 2 in \(\alpha\) but they are \textit{a priori} independent. There are many easy examples of forms with non-vanishing self-wedge, but vanishing determinant. The simplest example is, probably, \eqref{eq:SLSLform}.

In \autoref{def:GAreaMet}, the natural choice is \(|\det(\rho^*(\alpha))|^{1/14}\), since this must be non-zero for the non-degeneracy of the area metric. However \autoref{thm:GSplit} suggests that \(\qh(\alpha\wedge\alpha)\) is more natural, as it removes the coefficient in front of \(\alpha\).

\end{document}